\documentclass[11pt]{amsart}

\usepackage{amssymb}
\usepackage{amsmath,amssymb,epsf,wrapfig,multicol,epic,eepic,trig,mathrsfs,dsfont,color,graphicx}
 \usepackage[abs]{overpic}
 \usepackage{here}%図の読み込み場所をまさに箇所にする\begin{figure}[H]
\theoremstyle{plain}
\newtheorem{thm}{Theorem}
\newtheorem{cor}[thm]{Corollary}
\newtheorem{lem}[thm]{Lemma}
\newtheorem{prop}[thm]{Proposition}

\theoremstyle{remark}
\newtheorem{rem}[thm]{Remark}

\newcommand{\cal}{\mathcal}

\def\hsymbu#1{\smash{\lower1.7ex\hbox{\huge$#1$}}}

\def\rmoveio#1#2{%#2=2 oriented #2=1 non-oriented
\setlength{\unitlength}{#1}
\begin{picture}(50,30)
\put(5,0){\line(0,1){30}}

{\allinethickness{.8pt}
\put(10,15){\vector(1,0){13}}
\put(23,15){\vector(-1,0){13}}}

\qbezier(25,0)(25,20)(40,20)
\qbezier(40,20)(45,20)(45,15)
\qbezier(45,15)(45,10)(40,10)
\qbezier(40,10)(35,10)(31,14)
\qbezier(28,17)(25,25)(25,30)

\ifnum#2=2
\put(3,28){\path(0,0)(2,2)(4,0)}
\put(28,28){\path(0,0)(2,2)(4,0)}
\put(38,15){\makebox{${\Huge c_{1}}$}}
\fi

\end{picture}
}

\def\rmoveiio#1#2{%#2=2 oriented #2=1 non-oriented #2=3 cohearent oriented
\setlength{\unitlength}{#1}
\begin{picture}(60,30)
\put(5,0){\line(0,1){30}}
\put(15,0){\line(0,1){30}}

{\allinethickness{.8pt}
\put(20,15){\vector(1,0){15}}
\put(35,15){\vector(-1,0){15}}}

\qbezier(40,0)(42,1)(47,3)
\qbezier(52,6)(68,15)(52,24)
\qbezier(47,27)(42,30)(40,30)

\qbezier(60,0)(20,15)(60,30)

\ifnum#2=2
\put(2,27){\path(0,0)(3,3)(6,0)}
\put(12,27){\path(0,0)(3,3)(6,0)}
\put(40,27){\path(0,0)(0,3)(3,3)}
\put(60,27){\path(0,0)(0,3)(-3,3)}
\put(47,16){\makebox{${\Huge c_{1}}$}}
\put(47,10){\makebox{${\Huge c_{2}}$}}
\fi

\ifnum#2=3
\put(2,2){\path(0,0)(3,-3)(6,0)}
\put(12,27){\path(0,0)(3,3)(6,0)}
\put(40,3){\path(0,0)(0,-3)(3,-3)}
\put(60,27){\path(0,0)(0,3)(-3,3)}
\put(47,16){\makebox{${\Huge c_{1}}$}}
\put(47,10){\makebox{${\Huge c_{2}}$}}
\fi

\end{picture}
}
\def\rmoveiiio#1#2{%#2=2 oriented #2=1 non-oriented #2=3 cohearent oriented
\setlength{\unitlength}{#1}
\begin{picture}(75,30)
\put(0,0){\line(1,1){15}}
\qbezier(15,15)(20,20)(20,30)

\put(10,0){\line(-1,1){4}}
\qbezier(4,6)(-5,15)(5,25)
\put(5,25){\line(1,1){5}}

\qbezier(20,0)(20,10)(16,14)
\put(14,16){\line(-1,1){8}}
\put(4,26){\line(-1,1){4}}

{\allinethickness{.8pt}
\put(25,15){\vector(1,0){15}}
\put(40,15){\vector(-1,0){15}}}

\qbezier(50,0)(50,10)(55,15)
\put(55,15){\line(1,1){15}}

\put(60,0){\line(1,1){5}}
\qbezier(65,5)(75,15)(66,24)
\put(64,26){\line(-1,1){4}}

\put(70,0){\line(-1,1){4}}
\put(64,6){\line(-1,1){8}}
\qbezier(54,16)(50,20)(50,30)

\ifnum#2=2
\put(0,27){\path(0,0)(0,3)(3,3)}
\put(7,30){\path(0,0)(3,0)(3,-3)}
\put(17,27){\path(0,0)(3,3)(6,0)}

\put(10,23){\makebox{${\Huge c_{1}}$}}
\put(10,3){\makebox{${\Huge c_{2}}$}}
\put(20,13){\makebox{${\Huge c_{3}}$}}

\put(47,27){\path(0,0)(3,3)(6,0)}
\put(60,27){\path(0,0)(0,3)(3,3)}
\put(67,30){\path(0,0)(3,0)(3,-3)}

\put(70,23){\makebox{${\Huge c'_{2}}$}}
\put(70,3){\makebox{${\Huge c'_{1}}$}}
\put(60,13){\makebox{${\Huge c'_{3}}$}}
\fi

\end{picture}
}
\def\rmovevio#1#2{%#2=2 oriented #2=1 non-oriented
\setlength{\unitlength}{#1}
\begin{picture}(50,30)
\put(5,0){\line(0,1){30}}

{\allinethickness{.8pt}
\put(10,15){\vector(1,0){15}}
\put(25,15){\vector(-1,0){15}}}

\qbezier(30,0)(30,20)(45,20)
\qbezier(45,20)(50,20)(50,15)
\qbezier(50,15)(50,10)(45,10)
\qbezier(45,10)(30,10)(30,30)

\put(34,15){\circle{5}}

\ifnum#2=2
\put(3,28){\path(0,0)(2,2)(4,0)}
\put(28,28){\path(0,0)(2,2)(4,0)}
\put(38,15){\makebox{${\Huge c_{1}}$}}
\fi

\end{picture}
}

\def\rmoveviio#1#2{%#2=2 oriented #2=1 non-oriented #2=3 cohearent oriented
\setlength{\unitlength}{#1}
\begin{picture}(60,30)
\put(5,0){\line(0,1){30}}
\put(15,0){\line(0,1){30}}

{\allinethickness{.8pt}
\put(20,15){\vector(1,0){15}}
\put(35,15){\vector(-1,0){15}}}

\qbezier(40,0)(80,15)(40,30)
\qbezier(60,0)(20,15)(60,30)
\put(50,4){\circle{5}}
\put(50,25){\circle{5}}

\ifnum#2=2
\put(2,27){\path(0,0)(3,3)(6,0)}
\put(12,27){\path(0,0)(3,3)(6,0)}
\put(40,27){\path(0,0)(0,3)(3,3)}
\put(60,27){\path(0,0)(0,3)(-3,3)}
\put(47,15){\makebox{${\Huge c_{1}}$}}
\put(47,10){\makebox{${\Huge c_{2}}$}}
\fi

\ifnum#2=3
\put(2,2){\path(0,0)(3,-3)(6,0)}
\put(12,27){\path(0,0)(3,3)(6,0)}
\put(40,3){\path(0,0)(0,-3)(3,-3)}
\put(60,27){\path(0,0)(0,3)(-3,3)}
\put(47,15){\makebox{${\Huge c_{1}}$}}
\put(47,10){\makebox{${\Huge c_{2}}$}}
\fi

\end{picture}
}

\def\rmoveviiio#1#2{%#2=2 oriented #2=1 non-oriented #2=3 cohearent oriented
\setlength{\unitlength}{#1}
\begin{picture}(70,30)
\put(0,0){\line(1,1){15}}
\qbezier(15,15)(20,20)(20,30)

\put(10,0){\line(-1,1){5}}
\qbezier(5,5)(-5,15)(5,25)
\put(5,25){\line(1,1){5}}

\qbezier(20,0)(20,10)(15,15)
\put(15,15){\line(-1,1){15}}
\put(5,5){\circle{5}}
\put(15,15){\circle{5}}
\put(5,25){\circle{5}}

{\allinethickness{.8pt}
\put(28,15){\vector(1,0){14}}
\put(42,15){\vector(-1,0){14}}}

\qbezier(50,0)(50,10)(55,15)
\put(55,15){\line(1,1){15}}

\put(60,0){\line(1,1){5}}
\qbezier(65,5)(75,15)(65,25)
\put(65,25){\line(-1,1){5}}

\put(70,0){\line(-1,1){15}}
\qbezier(55,15)(50,20)(50,30)

\put(65,5){\circle{5}}
\put(55,15){\circle{5}}
\put(65,25){\circle{5}}

\ifnum#2=2
\put(0,27){\path(0,0)(0,3)(3,3)}
\put(7,30){\path(0,0)(3,0)(3,-3)}
\put(17,27){\path(0,0)(3,3)(6,0)}

\put(20,13){\makebox{${\Huge c_{1}}$}}

\put(47,27){\path(0,0)(3,3)(6,0)}
\put(60,27){\path(0,0)(0,3)(3,3)}
\put(67,30){\path(0,0)(3,0)(3,-3)}

\put(60,13){\makebox{${\Huge c'_{1}}$}}
\fi

\end{picture}
}

\def\rmovevivo#1#2{%#2=2 oriented #2=1 non-oriented #2=3 cohearent oriented
\setlength{\unitlength}{#1}
\begin{picture}(70,30)
\put(0,0){\line(1,1){15}}
\qbezier(15,15)(20,20)(20,30)

\put(10,0){\line(-1,1){5}}
\qbezier(5,5)(-5,15)(5,25)
\put(5,25){\line(1,1){5}}

\qbezier(20,0)(20,10)(16,14)
\put(14,16){\line(-1,1){14}}
\put(5,5){\circle{5}}
\put(5,25){\circle{5}}

{\allinethickness{.8pt}
\put(28,15){\vector(1,0){14}}
\put(42,15){\vector(-1,0){14}}}

\qbezier(50,0)(50,10)(55,15)
\put(55,15){\line(1,1){15}}

\put(60,0){\line(1,1){5}}
\qbezier(65,5)(75,15)(65,25)
\put(65,25){\line(-1,1){5}}

\put(70,0){\line(-1,1){14}}
\qbezier(54,16)(50,20)(50,30)

\put(65,5){\circle{5}}
\put(65,25){\circle{5}}

\ifnum#2=2
\put(0,27){\path(0,0)(0,3)(3,3)}
\put(7,30){\path(0,0)(3,0)(3,-3)}
\put(17,27){\path(0,0)(3,3)(6,0)}

\put(20,13){\makebox{${\Huge c_{1}}$}}

\put(47,27){\path(0,0)(3,3)(6,0)}
\put(60,27){\path(0,0)(0,3)(3,3)}
\put(67,30){\path(0,0)(3,0)(3,-3)}

\put(60,13){\makebox{${\Huge c'_{1}}$}}
\fi

\end{picture}
}
\def\abscrs#1{
\setlength{\unitlength}{#1}
\begin{picture}(75,30)
\allinethickness{1pt}
\put(0,5){\line(1,1){20}}
\put(0,25){\line(1,-1){8}}
\put(20,5){\line(-1,1){8}}
{\allinethickness{1pt}
\put(25,15){\vector(1,0){10}}
}
{\allinethickness{1pt}
\qbezier(50,3)(55,13)(60,3)
\qbezier(50,27)(55,17)(60,27)
\qbezier(42,10)(52,15)(42,20)
\qbezier(67,10)(57,15)(67,20)
}

{\allinethickness{2pt}
\put(45,5){\line(1,1){20}}
\put(45,25){\line(1,-1){8}}
\put(65,5){\line(-1,1){8}}
}

\end{picture}
}

\def\absvcrs#1{
\setlength{\unitlength}{#1}
\begin{picture}(75,30)
\allinethickness{1pt}
\put(0,5){\line(1,1){20}}
\put(0,25){\line(1,-1){20}}
\put(10,15){\circle{5}}
{\allinethickness{1pt}
\put(25,15){\vector(1,0){10}}
}

\put(40,5){\line(1,1){20}}
\put(40,25){\line(1,-1){10}}
\put(60,5){\line(-1,1){5}}

{\allinethickness{2pt}
\put(45,5){\line(1,1){20}}
\put(45,25){\line(1,-1){7}}
\put(65,5){\line(-1,1){7}}
}

\put(50,5){\line(1,1){20}}
\put(50,25){\line(1,-1){5}}
\put(70,5){\line(-1,1){10}}

\end{picture}
}
\def\abscpi#1{
\setlength{\unitlength}{#1}
\begin{picture}(35,30)
\allinethickness{1pt}
\put(5,5){\line(0,1){20}}
{%\allinethickness{.8pt}
\put(5,15){\circle*{3}}
}
{\allinethickness{1pt}
\put(12,15){\vector(1,0){8}}
}
{\allinethickness{2pt}
\put(30,5){\line(0,1){20}}
}
{\allinethickness{1pt}
\qbezier(27,5)(27,10)(30,15)
\qbezier(30,15)(33,20)(33,25)
\qbezier(33,5)(33,10)(31,13)
\qbezier(29,17)(27,20)(27,25)
}
\end{picture}
}

\def\cpmovei#1{
\setlength{\unitlength}{#1}
\begin{picture}(60,20)
\put(0,0){\line(1,1){20}}
\put(20,0){\line(-1,1){20}}
\put(10,10){\circle{5}}
\put(17,17){\circle*{2}}

{\allinethickness{.8pt}
\put(25,10){\vector(1,0){10}}
\put(35,10){\vector(-1,0){10}}}

\put(40,0){\line(1,1){20}}
\put(60,0){\line(-1,1){20}}
\put(50,10){\circle{5}}
\put(43,3){\circle*{2}}
\end{picture}
}

\def\cpmoveii#1{
\setlength{\unitlength}{#1}
\begin{picture}(40,20)
\put(5,0){\line(0,1){20}}
{%\linethickness{2pt}
\put(5.,7){\circle*{2}} %bar
\put(5.,13){\circle*{2}} }%bar
{\allinethickness{.8pt}
\put(15,10){\vector(1,0){10}}
\put(25,10){\vector(-1,0){10}}}
\put(35,0){\line(0,1){20}}
\end{picture}
}

\def\cpmoveiii#1{%#2=2 oriented #2=1 non-oriented #2=3 cohearent oriented
\setlength{\unitlength}{#1}
\begin{picture}(70,20)

\put(0,0){\line(1,1){20}}
\put(20,0){\line(-1,1){9}}
\put(9,11){\line(-1,1){9}}

\put(3,3){\circle*{2}}
\put(17,3){\circle*{2}}
\put(17,17){\circle*{2}}
\put(3,17){\circle*{2}}

{\allinethickness{.8pt}
\put(25,10){\vector(1,0){10}}
\put(35,10){\vector(-1,0){10}}}

\put(40,0){\line(1,1){20}}
\put(60,0){\line(-1,1){9}}
\put(49,11){\line(-1,1){9}}

\end{picture}
}

\begin{document}
%% title %%%%%%%%%%%%%%%%%%%%%%%%%%%%%%
%\begin{frontmatter}
\title[Coherent double coverings of  virtual link diagrams]
{Coherent double coverings of  virtual link diagrams}
\author{Naoko Kamada} %and Seiichi Kamada
\thanks{This work was supported by JSPS KAKENHI Grant Number 15K04879.}
%, 26287013.}
\address{ Graduate School of Natural Sciences,  Nagoya City University\\ 
1 Yamanohata, Mizuho-cho, Mizuho-ku, Nagoya, Aichi 467-8501 Japan
%Department of Mathematics, Osaka City University, \\
%Suimiyoshi,  Osaka 558-8585, Japan
}

\date{}

\begin{abstract} 
A virtual link diagram is called normal if the associated abstract link diagram is checkerboard colorable, and a virtual link is normal if it has a normal diagram as a representative. 
Normal virtual links have some properties similar to classical links.
%Virtual knot theory is a generalization of knot theory which is based on  Gauss chord diagrams and link diagrams on closed oriented surfaces. 
%A twisted knot is a generalization of a virtual knot, which corresponds to a link diagram on a possibly non-orientable surface. 
%A normal virtual link is presented by a virtual link diagram which is similar to a classical link diagram. The set of classical link diagrams is a subset of the set of normal virtual link diagrams.
In  this paper, we introduce a method of converting a virtual link diagram to a normal virtual link diagram. 
We show that the normal virtual link diagrams obtained by this method from two equivalent virtual link diagrams are equivalent.
 We discuss the relationship between this 
method and some invariants of virtual links.

\end{abstract}
\maketitle
%\begin{keyword}
%Virtual knot theory \sep Miyazawa polynomial
%\MSC Primary 57M25 \sep Secondary 57M27.
%\end{keyword}
%\end{frontmatter}

\section{Introduction}

%%%  1 virtual Link diagram%%%%%%%%%%%%%%%%%%%%%%%%
%L. H. Kauffman  \cite{rkauD} introduced virtual knot theory, which is a generalization of knot theory based on  Gauss  diagrams and link diagrams in closed oriented surfaces. 
Virtual links correspond to  stable equivalence classes of  links in thickened surfaces \cite{rCKS,rkk}. 
%Twisted knot  theory was introduced by Bourgoin. It is an extension of virtual knot theory. Twisted links correspond to stable equivalence classes of links in oriented 3-manifolds which are line bundles over (possibly non-orientable) closed surfaces  \cite{rBor,rCKS}. 
A virtual link diagram is called normal if the associated abstract link diagram is checkerboard colorable ($\S$ \ref{secnormal}). A virtual link is called normal or checkerboard colorable if it has a normal diagram as a representative. Every classical link diagram is normal, and hence 
the set of classical link diagrams is a subset of that of normal virtual link diagrams. The set of normal virtual link diagrams is a subset of that of virtual link diagrams. 
Jones polynomial is extend to virtual links \cite{rkauD}. It is shown in \cite{rkn0} that Jones polynomial of a normal virtual link has a property that Jones polynomial of a classical link has. 
For this property Khovanov homology is extended to normal virtual links in a natural way as stated in O. Viro \cite{rviro}.
The author introduced the method of  converting a virtual link diagram to a normal virtual link diagram by use of  the double covering technique in \cite{rkn2}. In this method, two converting virtual link diagrams obtained from two equivalent virtual link diagrams are related by generalized Riedemeister moves and K-flypes (See \cite{rkn2}). 

In  this paper, we introduce another method of converting a virtual link diagram to a normal virtual link diagram, which is called the coherent double covering technique. 
We show that two normal virtual link diagrams obtained from two equivalent virtual link diagrams by our method 
are equivalent as a virtual link.  In Section~\ref{sec:relationship}, we discuss the relationship between our method and some invariants of virtual links, as the odd writhe, the invariant of even ordered virtual link which is introduced by H. Miyazawa, K. Wada and Y. Yasuhara\cite{rMWY}.

%%%%%%%%%%%%%%%%%%%%%%%%%%%%%%%%
\section{Definitions and main results}\label{secnormal} 
A {\it virtual link diagram\/} is a generically immersed, closed and oriented 1-manifold in $\mathbb{R}^2$ with information of positive, negative or virtual crossing, on each double point.  A {\it virtual crossing\/} is an encircled double point without over-under information \cite{rkauD}. 
%A {\it twisted link diagram\/} is a virtual link diagram, possibly with {\it bars\/} on arcs \cite{rBor}. 
%Examples of a virtual link diagram and a twisted link  diagram are depicted in Figure~\ref{fig:extwtdiag} (i) and (ii), respectively.
%
%\begin{figure}[H]
%\centerline{
%\includegraphics[width=5cm]{ExTwist.eps}\hspace{1.2cm}
%\includegraphics[width=5cm]{Exabst.eps}}
%\centerline{(i)\hspace{2.3cm}(ii)\hspace{3cm}(iii)\hspace{2.3cm}(iv)}
%\caption{Examples of twisted link diagrams and abstract link diagrams}\label{fig:extwtdiag}
%\end{figure}
%
A {\it virtual link\/}  is an equivalence class of virtual link diagrams under Reidemeister moves and virtual Reidemeister moves depicted in Figure~ \ref{fgmoves}. We call Reidemeister moves and virtual Reidemeister moves  {\it generalized Reidemeister moves}. 
\begin{figure}[h]
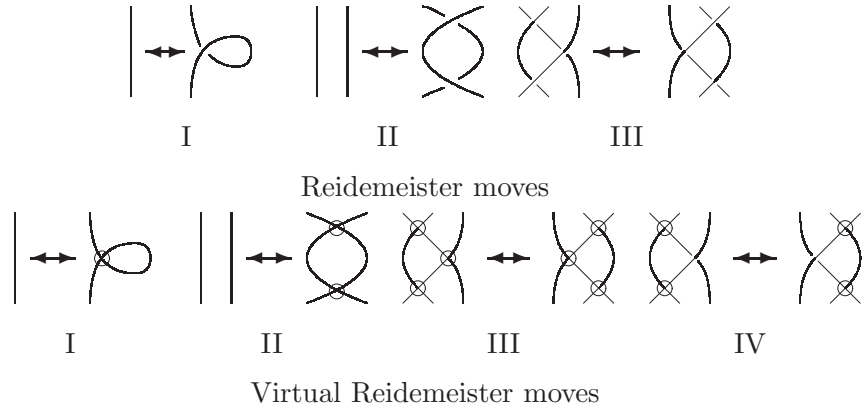

\centerline{
\begin{tabular}{ccc}
\rmoveio{.4mm}{1}&\rmoveiio{.4mm}{1}&\rmoveiiio{.4mm}{1}\\
I&II&III\\
\multicolumn{3}{c}{Reidemeister moves}
\end{tabular}}
\centerline{
\begin{tabular}{cccc}
\rmovevio{.4mm}{1}&\rmoveviio{.4mm}{1}&\rmoveviiio{.4mm}{1}&\rmovevivo{.4mm}{1}\\
I&II&III&IV\\
\multicolumn{4}{c}{Virtual Reidemeister moves}
\end{tabular}}
%\centerline{
%\begin{tabular}{ccc}
%\rmovetti{.5mm}&\rmovetiio{.5mm}&\rmovetiiio{.5mm}{1}\\
%I&II&III\\
%%\multicolumn{3}{c}{Twisted Reidemeister moves}
%\end{tabular}}
\caption{Generalized Reidemeister moves}\label{fgmoves}
\end{figure}

An {\it abstract link diagram} ({\it ALD})  is a pair $(\Sigma, D)$ 
of  a compact surface $\Sigma$ and a link diagram $D$ on $\Sigma$ such that the underlying 4-valent graph 
$|D|$ is a deformation retract of $\Sigma$. 

%
%\begin{figure}[h]
%\centerline{
%\includegraphics[width=6cm]{Exabst.eps}
%}
%\centerline{(i)\hspace{3cm}(ii)}
%\caption{Example of an ALD}\label{fg:ExALD}
%\end{figure}
%
We obtain an ALD from a virtual link diagram $D$ by corresponding a diagram to an ALD as in Figure~\ref{fg:virabs} (i) and (ii).
Such an  ALD is called the {\it ALD associated with} $D$.
Figure~\ref{fig:extwtdiag} shows a virtual link diagram and the ALD associated with it. 
For details on abstract link diagrams and their relations to virtual links, refer to \cite{rkk}.
\begin{figure}[h]
\centerline{
\begin{tabular}{ccc}
\abscrs{.5mm}&\absvcrs{.5mm}&\abscpi{.5mm}\\
(i)&(ii)&(iii)
\end{tabular}
}
\caption{The correspondence from a virtual link diagrams to an ALD}\label{fg:virabs}
\end{figure}
\begin{figure}[h]
\centerline{
\begin{tabular}{cc}
\includegraphics[width=2.3cm]{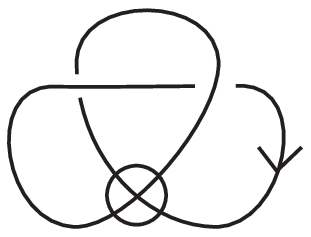}&
\includegraphics[width=2.3cm]{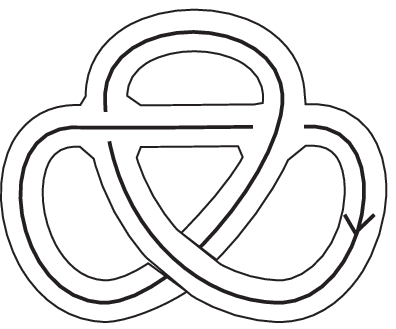}\\
(i)&(ii)
%&(iii)&(iv)\\
\end{tabular}}
\caption{A virtual link diagram and  an ALD}\label{fig:extwtdiag}
\end{figure}
Let $D$ be a virtual link diagram and $(\Sigma, D_{\Sigma})$ the ALD associated with $D$. 
The diagram $D$ is said to be {\it normal} or {\it checkerboard colorable} if the regions of $\Sigma-|D_{\Sigma}|$ can be colored black and white such that colors of two adjacent regions are different. In Figure~\ref{fig:exccdiag}, we show an example of a normal diagram. (The orientations of the diagrams in Figure~\ref{fig:exccdiag} are alternate orientations, which are discussed later.) A classical link diagram is normal. A virtual link is said to be {\it normal} if it has a normal virtual link diagram. Note that normality is not necessary to be 
preserved  under generalized Reidemeister moves. 
For example the virtual link diagram in the right of Figure~\ref{fgexdiag} is not normal and is equivalent to the trefoil knot diagram in the left which is normal. 
\begin{figure}[h]
\centerline{
\includegraphics[width=7cm]{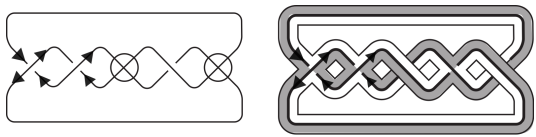}}
\caption{A normal twisted link  diagram and its associated ALD with a checkerboard coloring}\label{fig:exccdiag}
\end{figure}
\begin{figure}[h]
\centerline{
\includegraphics[width=5cm]{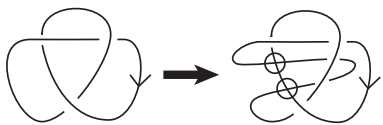}}
\caption{A diagram of a normal virtual link which is not normal}\label{fgexdiag}
\end{figure}

%We recall from \cite{rkk7} the double covering diagram of a twisted link diagram. Let $D$ be a twisted link diagram. 
H. Dye introduced the notion of cut points on a virtual link diagram in \cite{rDye1}.
%her talk presented in the Special Session 35, ``Low Dimensional Topology and Its Relationships with Physics", held in Porto, Portugal, June 10-13, 2015 as part of the 1st AMS/EMS/SPM Meeting. 

Let $(D, P)$ be a pair of a virtual link diagram $D$ and a finite set $P$ of points on edges of $D$. 
We obtain an  ALD from $(D, P)$ as in Figure \ref{fg:virabs} (i), (ii) and (iii) and  call it the {\it ALD associated with} $(D, P)$.
 See Figure~\ref{fgcutpt} (ii) and (iii). 
If the ALD associated with $(D, P)$ is normal,  then we call the set of points $P$  a {\it cut system} of $D$ and call each point of $P$ a {\it cut point}.
%Such a normal ALD is said to be a {\it normal ALD} associated with $(D, P)$. 
For  the virtual link diagram  in Figure~\ref{fgcutpt} (i) we show an example of a cut system  in Figure~\ref{fgcutpt} (ii)  and the  ALD associated with it with a checkerboard coloring in Figure~\ref{fgcutpt} (iii). 

%
%\begin{figure}[h]
%\centerline{
%\begin{tabular}{ccc}
%\abscrs{.4mm}&\absvcrs{.4mm}&\abscpi{.4mm}
%\end{tabular}
%}
%\caption{The correspondence from virtual link diagrams to ALD}\label{fg:virptabs}
%\end{figure}
%

%
\begin{figure}[h]
\centerline{
\includegraphics[width=8cm]{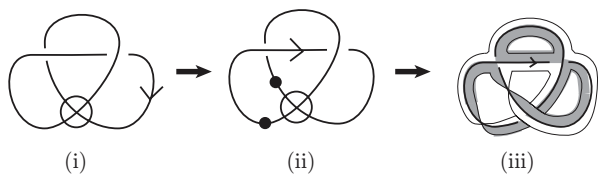}}
\caption{Example of cut points}\label{fgcutpt}
\end{figure}

A virtual link  diagram is said to admit an {\it alternate orientation} if it can be given an orientation  such that an  orientation of an edge switches at each classical crossing  as in Figure~\ref{fgalternateorip}. 
(Here an edge means an edge of $|D|$, where endpoints are classical crossings and there might be some virtual crossings on it.) 
The virtual link diagram in Figure~\ref{fig:exccdiag} admits an alternate orientation. 
It is known that a virtual link diagram is normal if and only if it admits an alternate orientation \cite{rkns}. 
%
%\begin{figure}[h]
%\centerline{
%\includegraphics[width=5cm]{altorientationv.eps}
%}
%\caption{Alternate orientation}\label{fig:altori}
%\end{figure}
%
%\begin{prop}{\cite{rkn1}}\label{prop1}
%Let $D$ be a twisted link diagram. $D$ admits an alternate orientation if and only if $D$ is normal.
%\end{prop}
%
\begin{figure}[h]
\centerline{
\includegraphics[width=6cm]{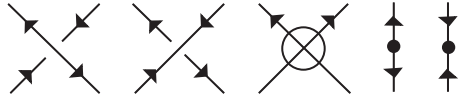}}
\caption{Alternate orientation}\label{fgalternateorip}
\end{figure}
Note that a finite set $P$ of points on $D$ is a cut system if and only if $(D,P)$ admits an alternate orientation such that the orientations  of edges are as in Figure~\ref{fgalternateorip} at each classical crossing of $D$ and each point of $P$ (cf. \cite{rkn2,rkns}).

The {\it canonical cut system} of a virtual link diagram $D$  is the set of points 
 that is obtained by 
giving two points in a neighborhood of each virtual crossing of $D$ as in Figure~\ref{fgVtoCC} (i). 
\begin{figure}[h]
\centerline{
\includegraphics[width=1.5cm]{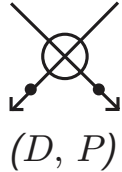}}
%\centerline{$D$\hspace{1cm}$(D,P)$\hspace{1cm}\parbox{1.5cm}{\small an ALD associated with $D_C$}\hspace{.5cm}\parbox{1.5cm}{\small an ALD associated with $(D,P)$}}
\caption{The canonical cut system of a virtual link diagram}\label{fgVtoCC}
\end{figure}
\begin{prop}[\cite{rkn2}]\label{propcan}
The canonical cut system is a cut system.
\end{prop}

%\begin{proof}
%For a virtual link diagram $D$, let $D_C$ be a classical link diagram which is obtained from $D$ by replaceing all virtual crossings of $D$ with classical ones. Note that there is a checkerboard coloring for the ALD %associated with $D_C$. At each classical crossing, the checkerboard coloring is as in Figure~\ref{fgVtoCC} (ii). 
%Let $P$ be the canonical cut system of $D$. 
%The ALD associated with $(D, P)$  is checkerboard colorable such that its coloring is inherited from that of $D_C$ as in Figure~\ref{fgVtoCC} (ii) and (iii).
%\end{proof}

Dye introduced cut point moves  depicted in Figure~\ref{fgcutptmove}.
Then we have the following.

\begin{figure}[h]
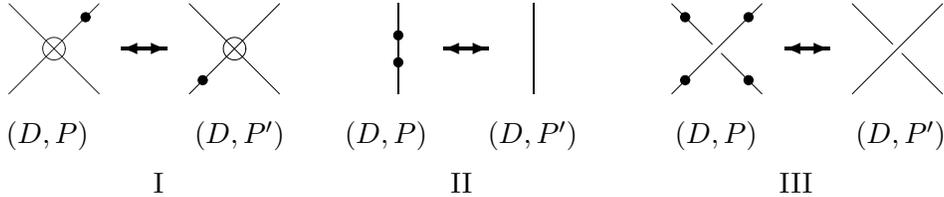

\centerline{
\begin{tabular}{ccc}
\cpmovei{.6mm}\quad\quad&\cpmoveii{.6mm}&\quad\quad\cpmoveiii{.6mm}\\
$(D,P)$\hspace{1.4cm}$(D, P')$\phantom{M}&$(D,P)$\hspace{.8cm}$(D, P')$&\quad$(D,P)$\hspace{1.3cm}$(D, P')$\\
I&II&III\\
%(i)&(ii)&(iii)\\
\end{tabular}
}
\caption{Cut point moves}\label{fgcutptmove}
\end{figure}
\begin{thm}[\cite{rkn2}]\label{thmc1}
For a virtual link diagram $D$, two cut systems of  $D$ are related by a sequence of cut point moves I, II and III.
\end{thm}
%
%\begin{proof}%[Proof of Theorem~\ref{thmc1}]
%Let $P$ and $P'$ be two cut systems of $D$. 
%For $(D,P)$ and $(D,P')$, give  alternate orientations ${\cal O}$ and ${\cal O}'$, respectively. 
%Let $c_1, c_2, \dots , c_m$ be classical crossings of $D$ where the orientations of edges of $\cal O$ are different from those of 
%$\cal O'$. 
%Apply cut point moves III  at $c_1, c_2, \dots , c_m$ to $(D,P)$, then we obtain a cut system $P''$ of $D$. 
%There is an alternate orientation of $(D,P'')$, say $\cal O''$,  such that each classical crossing in $D$ admits the same 
%orientation to that of $\cal O'$. 
%Applying some cut point moves I to $P''$, we have a cut system $P'''$ such that for each edge $e$ of $D$, the number of cut points on $e$ in $(D,P''')$ is congruent to that in $(D,P')$ modulo $2$. 
%Note that, for each edge $e$ of $|D|$, the number of cut points of $P''$ on $e$ is congruent to that of $P'$ modulo $2$. 
%Thus $(D,P')$ is obtained from $(D,P'')$ by cut point moves I and II. 
%\end{proof}

\begin{cor}[\cite{rDye1}, c.f. \cite{rkn2}]\label{cor1}
For any virtual link diagram with a cut system,  the number of cut points is 
even.
\end{cor}
%\begin{proof}
%The number of cut points of  the canonical cut system is even. Since cut point moves do not change the parity of the number of cut points, we obtain the result.
%\end{proof}

Let $D$ be a virtual link diagram with a cut system $P$. 
Assume that $(D, P)$ is on the right of the $y$-axis in the $xy$-plane and all crossings and cut points have distinct $y$-coordinates.  Let  $ (D^*, P^*)$ be a copy of $(D,P)$ on the left of the $y$-axis in the $xy$-plane which is obtained from $(D,P)$ by sliding along the $x$-axis. Let $\{p_1,\dots, p_k\}$ be the set of cut points of $P$ and for $i\in\{1,\dots, k\}$, we denote  by $p_i^*$ the cut point of $D^*$ corresponding to $p_i$.
%The  cut system of $D^*$ is obtained from $P$ by the reflection. We denote it by $P^*$.
See Figure~\ref{fgconvert1} (i). 
\begin{figure}[h]
\centering{
\includegraphics[width=5.5cm]{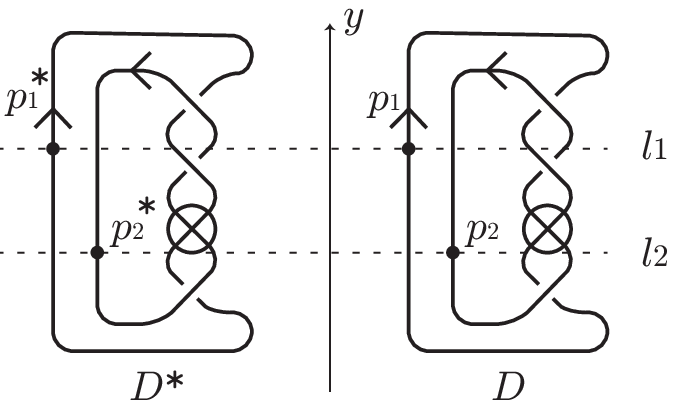}\hspace{1cm}
\includegraphics[width=4.8cm]{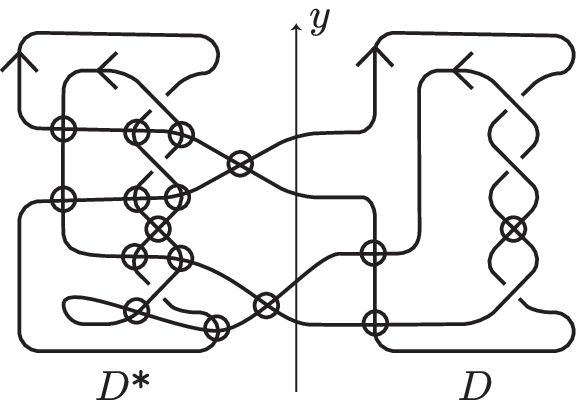}}\\
\centerline{(i)\hspace{6cm}(ii)}
%\caption{A virtual link diagram and it's refrection}
\caption{The coherent double covering of a virtual link diagram}\label{fgconvert1}
\end{figure}
For horizontal lines $l _1, \dots , l_k$ such that  $l_i$ contains $p_i$ and the corresponding cut point $p_i^*$ of $D^*$, 
we  replace  each part of $D\amalg D^*$ in a neighborhood of $N(l_i)$ for each $i\in\{1,\dots, k\}$  as in Figure~\ref{fgconvert2}. We denote by $\phi(D, P)$ the virtual link diagram obtained this way.

For example, for the virtual link diagram $D$ with the cut system $P$  depicted as in Figure~\ref{fgconvert1} (i), the virtual link diagram $\phi(D, P)$ is as in Figure~\ref{fgconvert1} (ii).
%%%%%%%%%%%%%%%%%%%
\begin{figure}[h]
\centering{
\includegraphics[width=10cm]{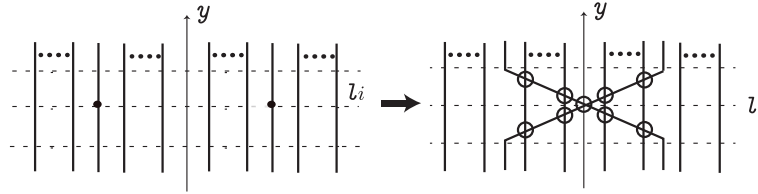}}
\caption{The replacement of diagrams}\label{fgconvert2}
\end{figure}
%%%%%%%%%%%%%%%%%%%%%%
%%%%%%%%%%%%%%%%%%%%
%\begin{figure}[h]
%\centering{
%\includegraphics[width=7cm]{Exconv3.eps} }
%\caption{The image by $\phi$}\label{fgconvert3}
%\end{figure}
Then we have the following.

%\begin{thm}{\cite{rkk7}}\label{kkthm1}
%Let $D_1$ and $D_2$ be twisted link diagrams.
%If $D_1$ and $D_2$ are equivalent as a twiated link, then $\psi(D_1)$ and $\psi(D_2)$ are equivalent as a virtual link.
%\end{thm}

%
\begin{prop}\label{thm1}
For a virtual link diagram $D$ with a cut system $P$, $(D, P)$,   $\phi(D, P)$ is normal.
\end{prop}
\begin{proof}
Let $D$ be a virtual  link diagram $D$ with a cut system $P$. The virtual link diagram $D^*$ with the cut system $P^*$ is a copy of  $(D, P)$ as the previous manner as in Figure~\ref{fgconvert4} (i). The  ALDs obtained from $D$ and $D^*$ can be colored as in Figure~\ref{fgconvert4} (ii). Then we see that $\phi(D, P)$ is normal as in the right of Figure~\ref{fgconvert4} (ii). 
\end{proof}
%%%%%%%%%%%%%%%%%%%%%%%
\begin{figure}[h]
\centering{
\includegraphics[width=11cm]{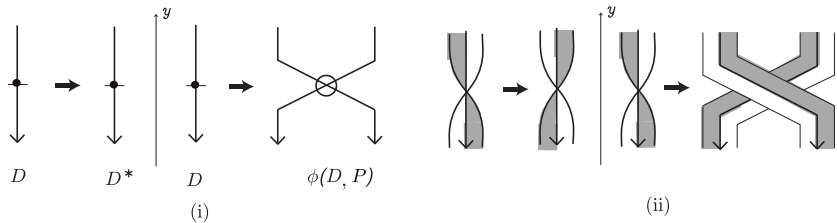} }
\caption{The converting diagram}\label{fgconvert4}
\end{figure}
%%%%%%%%%%%%%%%%%%

%We define the map $t$ from the set of virtual link diagrams to that of twiated link diagrams as follows;
%Let $D$ be a virtual link diagram. The image of $D$ under the map $t$ is obtained from $D$ by adding two bars around each virtual crossing of $D$ depicted as in 
%Figure~\ref{fgmapvt} (i).
%
%\begin{figure}[h]
%\centerline{
%\begin{tabular}{cc}
%\includegraphics[width=3cm]{MapVtoT2.eps}&\includegraphics[width=6cm]{MapVtoCCT2.eps}\\
%(i)&(ii)\\
%\end{tabular}}
%\caption{A twisted link diagram $t(D)$}\label{fgmapvt}
%\end{figure}

%\begin{thm}\label{thm2}
%For a virtual link diagram $D$, $t(D)$ and $\phi(t(D))$ are normal.
%\end{thm}

%
%\begin{thm}\label{thm1}
%Let $D_1$ and $D_2$ be virtual link diagrams.
%If $D_1$ and $D_2$ are equivalent (or K-equivalent), then $\phi(t(D_1))$ and $\phi(t(D_2))$ are K-equivalent.
%\end{thm}

%%%%%%%%%%%%%%%%%%%%%%%%%%%%%%%%%%%%
%\section{Map from a virtual link diagram to a normal virtual link diagram}\label{map}
%We define a map $\phi$ from a virtual link diagram with a cut system to a virtual link diagram as follows:
%Let $D$ be a virtual link diagram and $P$ be a cut system of $D$. 
%We denote cut  points of $P$ by $p_1,\dots , p_k$. 
%Assume that $(D,P)$ is on the right of the $y$-axis and all cut points are parallel to the $x$-axis with disjoint $y$-coordinates. Let  $D^*$ be  the virtual link diagram obtained from $D$ by reflection with respect to the $y$-axis and switching all real crossings of $D$. The  cut system of $D^*$ is obtained from $P$ by the reflection. We denote it by $P^*$.
%See Figure~\ref{fgconvert1} (i). 

Then we have 
$$\phi:\{\mbox{virtual link diagrams with cut systems}\}\longrightarrow\{\mbox{normal virtual link diagrams}\}.$$

\begin{rem}
Precisely speaking, $\phi(D,P)$ is well-defined when $(D,P)$ satisfies that all crossings and cut points have distinct $y$-coordinates. If we change $(D,P)$ by an isotopy of $\mathbb{R}^2$, $\phi(D,P)$ is preserved under an isotopy of $\mathbb{R}^2$ and virtual Reidemeister moves.
\end{rem}

Let  $(D,P)$ be a virtual link diagram with a cut system $P$. We call $\phi(D,P)$ the {\it converted normal diagram} or the {\it coherent double covering diagram} of $(D, P)$. 
Let $D$ be a virtual link diagram. For the canonical cut system $P_0$ of $D$, we denote by $\phi_0(D)$ the converted normal diagram $\phi(D, P_0)$. 
We call $\phi_0(D)$ the {\it canonical converted normal diagram} of $D$.
Then we have 
$$\phi_0:\{\mbox{virtual link diagrams}\}\longrightarrow\{\mbox{normal virtual link diagrams}\}.$$

%The local replacement of a virtual link diagram depicted in Figure~\ref{fgkflype} is called a {\it Kauffman flype} or a {\it K-flype}.
% If a virtual link diagram $D'$ is obtained from $D$ by a finite sequence of generalized Reidemeister moves and K-flypes, then they are said to be {\it K-equivalent}.
%
%\begin{figure}[h]
%\centering{
%\includegraphics[width=3cm]{kflype1.eps} }
%\caption{Kauffman flype}\label{fgkflype}
%\end{figure}
%
%\begin{rem}\label{rem:kflype}
%The $f$-polynomials of K-equivalent virtual link diagrams are the same \cite{rkauD}. 
%For a virtual link diagram of $D$, if a virtual link diagram $D'$ is obtained from $D$ by a K-flype at a classical crossing $c$, then the sign of the corresponding classical crossing $c'$ of $D'$ is the same as that of $c$. 
%If $D$ is normal, then $D'$ is normal. 
%\end{rem}

The following is our main theorem.
\begin{thm}\label{thm2}
Let $(D,P)$ and $(D',P')$ be virtual link diagrams with cut systems. If $D$ and $D'$ are equivalent, 
then the converted normal diagrams $\phi(D,P)$ and $\phi(D',P')$ are equivalent. 
In particular, if $D$ and $D'$ are equivalent, 
then the canonical converted normal diagrams $\phi_0(D)$ and $\phi_0(D')$ are equivalent.
\end{thm}

%\begin{cor}\label{cor6}
%For a virtual knot $K$ and its diagram $D$ with a cut system $P$,  $\mathrm{lk}(\phi(D,P))$ is an invariant of virtual knots.
%\end{cor}
%%%%%%%%%%%%%%%%%%%%%%%%%%

\section{Proof of Theorem \ref{thm2}}
Theorem~\ref{thm2} is obtained from Lemmas~\ref{lem2} and \ref{lem1} stated below. 
\begin{lem}\label{lem2}
Let $D$ be a virtual  link diagram. Suppose that $P$ and  $P'$ are cut systems of $D$.
Then the converted normal diagrams $\phi(D,P)$ and $\phi(D,P')$ are equivalent.
\end{lem}
\begin{proof}%[Proof of Lemma \ref{lem2}]
Let $D$ be a virtual  link diagram with a cut system $P$. Suppose that  $P'$ is a cut system of $D$ obtained from $P$ by one of cut point moves I , II or III in Figure~\ref{fgcutptmove}. 
Then $\phi(D, P')$ is related to $\phi(D, P)$ by  virtual Reidemeister moves II and III as in Figure~\ref{fgprflem21}. Thus $\phi(D, P)$ and $\phi(D, P')$ are equivalent. 
\end{proof}
\begin{figure}[h]
\centerline{
\begin{tabular}{cc}
 \includegraphics[width=5cm]{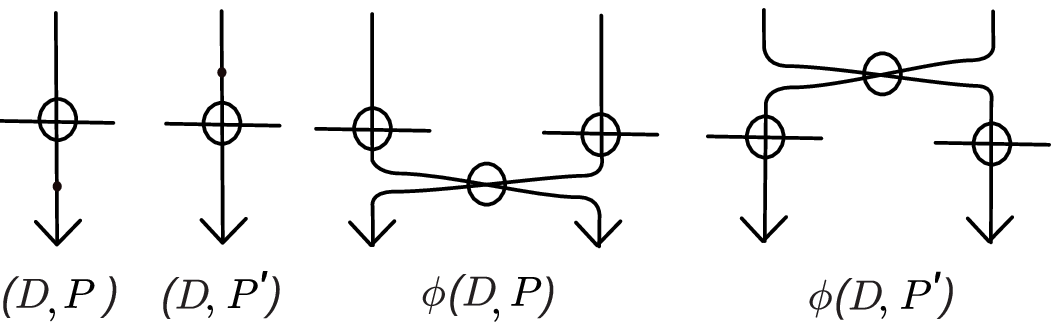}&
  \includegraphics[width=5cm]{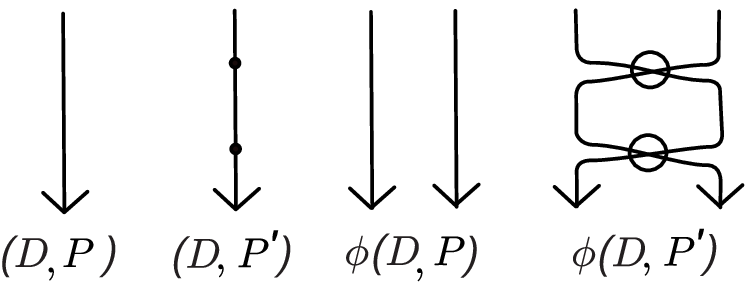}\\
 (i)&(ii)\\   
\end{tabular}
}
\centerline{
\includegraphics[width=7cm]{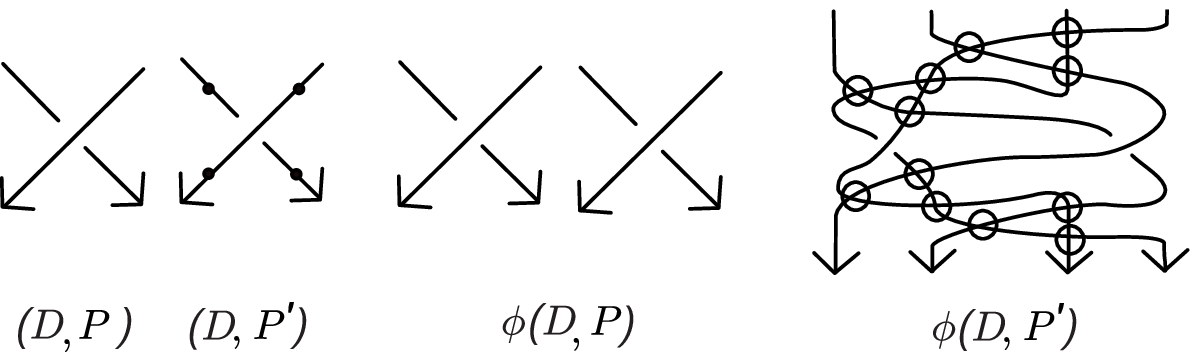}
}\centerline{(iii)}
%\centering{
%\includegraphics[width=7cm]{prflem23.eps}}\\
%\centering{(iii)}
\caption{Converted normal diagrams related by a cut point move}\label{fgprflem21}
\end{figure}
\begin{lem}\label{lem1}
Let $D_1$ and $D_2$ be virtual link diagrams.
If $D_1$ and $D_2$ are  equivalent then the canonical converted normal virtual link diagrams $\phi_0(D_1)$ and $\phi_0(D_2)$ are equivalent.
\end{lem}
\begin{proof}%[Proof of Lemma \ref{lem1}]
Let $D_1$ be a virtual  link diagram with the canonical cut system $P_1$. Suppose that a virtual link diagram $D_2$ is obtained from $D_1$ by one of generalized Reidemeister moves   
and $P_2$ is the canonical cut system of $D_2$.
If $D_2$ is related to $D_1$ by one of Reidemeister moves, then $\phi_0(D_1)=\phi(D_1, P_1)$ and $\phi_0(D_2)=\phi(D_2,P_2)$ are related by two Reidemeister moves. 
%Suppose that $D_2$ is related to $D_1$  by one of virtual Reidemeister moves as in Figure~\ref{fgprflem11} (i), (ii), (iii) and (iv). 
%As in Figure~\ref{fgprflem11} (i), if $D_2$ is related to $D_1$  by a virtual Reidemeister move I, then $\phi(D_1,P_1)$ is related to $\phi(D_2,P_2)$ by some virtual Reidemeister moves. 
As in Figure~\ref{fgprflem11} (i) (or (ii)), suppose that  $D_2$ is related to $D_1$  by a virtual Reidemeister move I (or II) and let $P_2'$ be the cut system obtained from $P_2$ by cut point moves I and II as in the figure. 
By Lemme~\ref{lem2},  $\phi(D_2,P_2)$ and $\phi(D_2,P_2')$ are equivalent.
On the other hand $\phi(D_1,P_1)$ and $\phi(D_2,P_2')$ are related by two virtual Reidemeidter moves I (or II) as in Figure~\ref{fgprflem11} (i) (or (ii)). Thus $\phi(D_1,D_1)$ and  $\phi(D_2,D_2)$ are equivalent.
As in Figure~\ref{fgprflem11} (iii), suppose that  $D_2$ is related to $D_1$  by a virtual Reidemeister move III and let $P_1'$ (or $P_2'$) be the cut system obtained from $P_1$ (or $P_2$) by cut point moves I and II as in the figure. 
By Lemme~\ref{lem2},  $\phi(D_1,P_1)$  (or $\phi(D_2,P_2)$ ) and $\phi(D_1,P_1')$ (or $\phi(D_2,P_2')$) are equivalent.
On the other hand, $\phi(D_1,P_1')$ and $\phi(D_2,P_2')$ are  related by two virtual Reidemeidter move III. 
As in Figure~\ref{fgprflem11}  (iv), suppose that  $D_2$ is related to $D_1$  by a virtual Reidemeister move IV and let $P_1'$ (or $P_2'$) be the cut system obtained from $P_1$ (or $P_2$) by cut point moves I and II  
(or cut point moves I , II and III) as in the figure. By Lemme~\ref{lem2},  $\phi(D_1,P_1)$  (or $\phi(D_2,P_2)$ ) and $\phi(D_1,P_1')$ (or $\phi(D_2,P_2')$) are equivalent.
On the other hand, 
as in Figure~\ref{fgprflem11}  (v),  $\phi(D_1',P_1')$  and $\phi(D_2',P_2')$ are equivalent by virtual Reidemeister moves. 

In Figure~\ref{fgprflem11}, if the orientation of some strings of virtual link diagram $D_i$ are different from those of it,  we have the result by a similar argument. 
\begin{figure}[h]
\centering{
\begin{tabular}{ccc}
\includegraphics[width=3.cm]{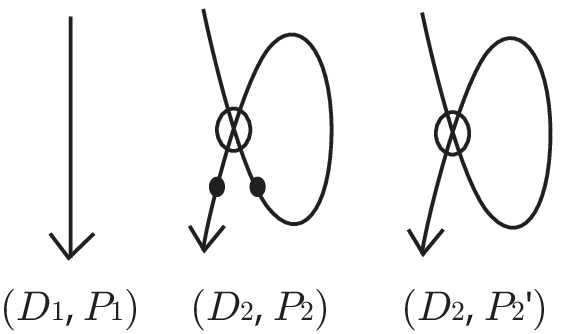} &\includegraphics[width=4.cm]{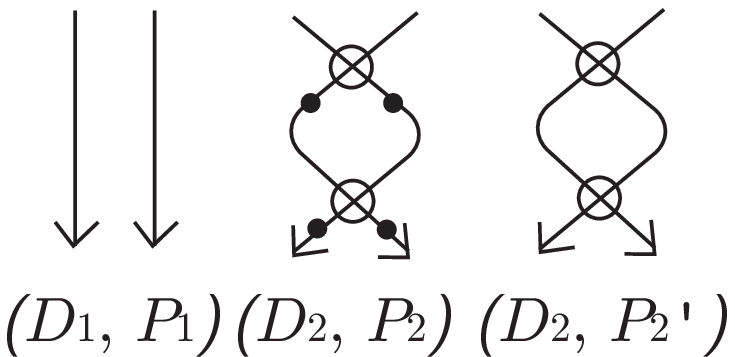}&\includegraphics[width=5.cm]{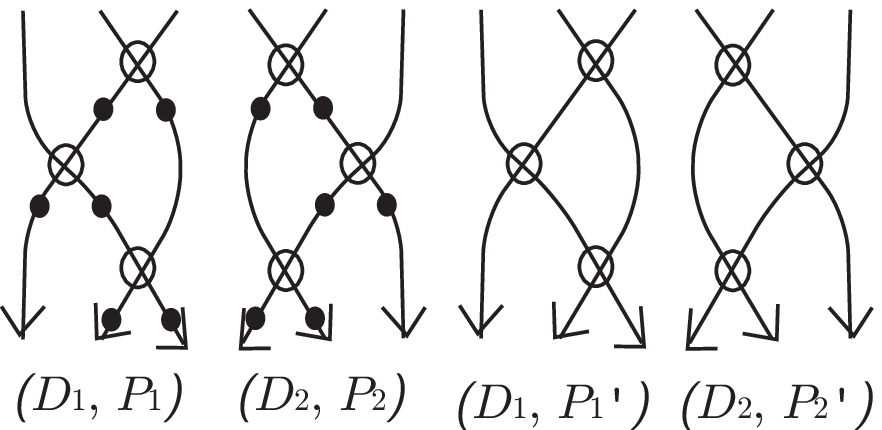}\\
{\small Virtual Reidemeister move I }&{\small Virtual Reidemeister move II}
&{\small Virtual Reidemeister move III}\\
(i)&(ii)&(iii)\\
\end{tabular}}
\centering{
\begin{tabular}{cc}
\includegraphics[width=6cm]{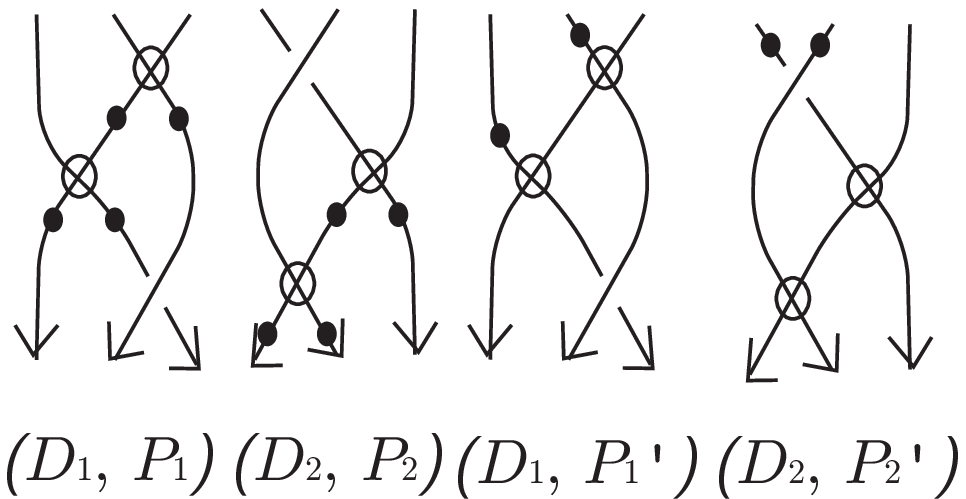}
&\includegraphics[width=7cm]{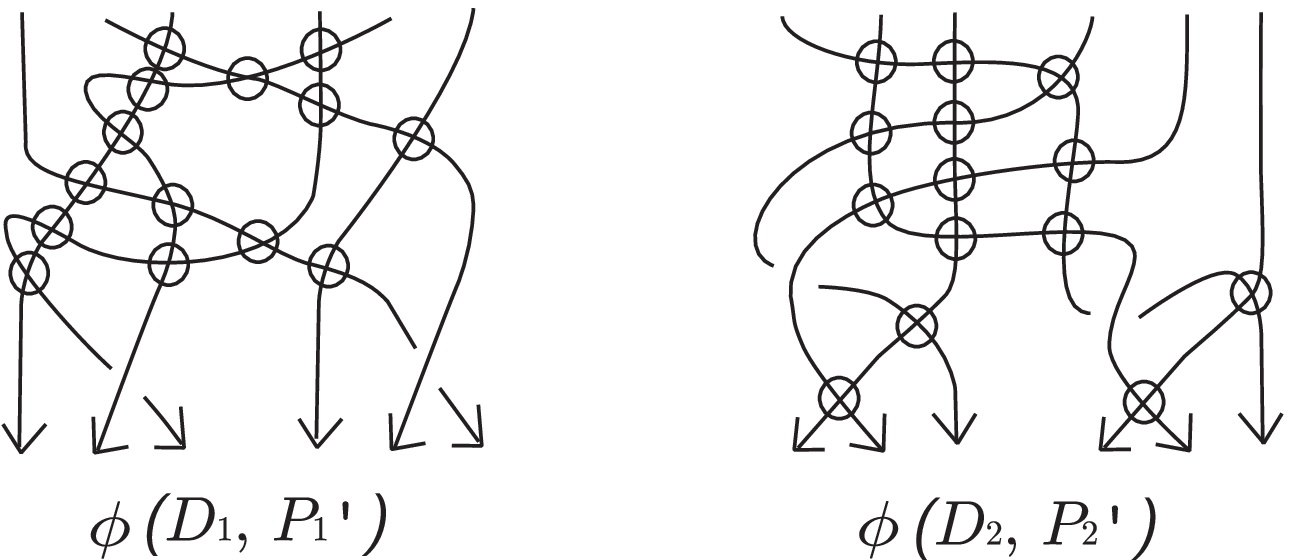}\\
%\multicolumn{2}{c}{\includegraphics[width=7cm]{prflem1k.eps}}\\
{\small Virtual Reidemeister move IV}
&\parbox{6cm}{\small Coherent double covering diagrams (virtual Reidemeister move IV)}\\
(iv)
&(v)\\
\end{tabular}}
\caption{Diagrams related by a virtual Reidemeister move}\label{fgprflem11}
\end{figure}
%
%If $D_1$ is related to $D_2$ by a K-flype, then $\phi(D_1, P_1)$ and $\phi(D_2,P_2)$ are related by some virtual Reidemeister moves in  Figure~\ref{fgprflem11} (v). 

\end{proof}
%

%%%%%%%%%%%%%%%%%%%%%%%%%%%%%%%%%%%%%
\section{A relationship between invariants and coherent double covering diagrams}\label{sec:relationship}
%\section{The odd writhe of a virtual knot}\label{sec:oddwrithe}
In this section we discuss a relationship between some invariants and coherent covering diagrams.

\subsection{Odd writhe}
We give  
an interpretation of the odd writhe of a virtual knot in terms of the linking number of the converted normal diagram $\phi(D, P)$. The argument of this subsection is  similar to that of Section 4 in \cite{rkn2}. 

For a 2-component virtual link diagram  $D$, the half of the sum of signs of nonself classical crossings of $D$ is called the {\it linking number} of $D$.  % (cf. Remark~\ref{rem:kflype}).
The linking number is an invariant of 2-component virtual link.

We have the following theorems (Theorems \ref{thmk1} and \ref{thmk2}).

\begin{thm}\label{thmk1}
Let $(D,P)$ be a virtual knot diagram with a cut system. 
Then $\phi(D,P)$ is a 2-component virtual link diagram and the linking number of $\phi(D,P)$ is an invariant of the virtual knot represented by $D$.
\end{thm}

The odd writhe is a numerical invariant of virtual knots \cite{rkauE}. We recall the definition of the odd writhe later. 

\begin{thm}\label{thmk2}
Let $(D,P)$ be a virtual knot diagram with a cut system. The linking number of $\phi(D,P)$ is equal to the odd writhe of $D$.
\end{thm}

Let $D$ be a virtual link diagram. The {\it Gauss diagram} of $D$ is a set of oriented circles which are the preimage of $D$ with oriented chords each of which corresponds to a classical crossing and  its tail (or its head) indicates an overpass  (or an underpass) of the  classical crossing. Each chord is equipped with a sign of the corresponding classical crossing. 
The Gauss diagram in Figure \ref{fgexgauss1} (i)  is that of the virtual knot diagram in Figure \ref{fig:extwtdiag} (i). 
\begin{figure}[h]
\centering{
\includegraphics[width=10cm]{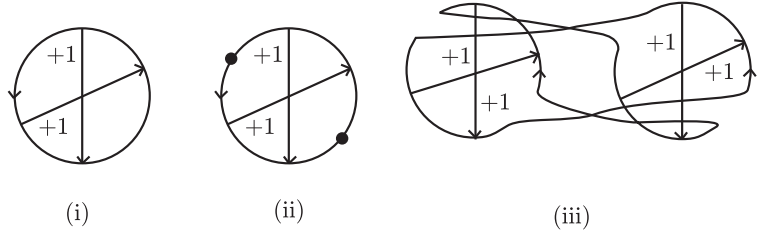} }
\caption{The Gauss diagrams}\label{fgexgauss1}
\end{figure}
For  a virtual link diagram $D$ with a cut system $P$, the Gauss diagram with points, denoted by $G(D,P)$, of $(D,P)$ is obtained from the Gauss diagram of $D$ by adding points on arcs which correspond to the points of $P$. We denote by $G(D)$ the Gauss diagram of $D$ and by $G(P)$ the points of $G(D,P)$ corresponding to $P$. Then $G(D, P)=(G(D), G(P))$. In Figure~\ref{fgexgauss1} (ii) we show  the Gauss diagram with  points of the virtual link diagram with points in Figure~\ref{fgcutpt} (ii).
Let $(D^*,P^*)$ be the virtual link diagram with a cut system which is a copy of  $(D,P)$.
The Gauss diagram with points $G(D^*,P^*)$ is a copy of  $G(D,P)$.  
%See  Figure~\ref{fgexgauss1} (ii) and (iii). 
The Gauss diagram of $\phi(D,P)$ is obtained from the Gauss diagrams of $D\amalg D^*$ by a local replacement around each point $p\in G(P)$ of the Gauss diagram $G(D)$ and around the corresponding point $p^*\in G(P^*)$ of $G(D^*)$ as in  Figure~\ref{fggaussrep}.  For a virtual knot diagram of $D$ with a set of points on edges $P$, we denote the Gauss diagram of $\phi(D,P)$ by $G(\phi(D, P))$. 
The Gauss diagram in Figure~\ref{fgexgauss1} (iii) is $G(\phi(D,P))$ for $(D, P)$ depicted in Figure~\ref{fgcutpt} (ii) by the map $\phi$. 
%The Gauss diagram in Figure~\ref{fgexgauss1} (iv) is the Gauss diagram of the converting diagram of the virtual link diagram with points on its edges in Figure~\ref{fgcutpt} (ii). 
%
%Let $D$ be a virtual link diagram with a cut system $P$. We denote the Gauss diagram of $D$ with the preimage of cut system $P$ by $G(D,P)$. In Figure~\ref{fgexgauss1} (ii) we see  the Gauss diagram with a cut system of the virtual link diagram with a cut system in Figure~\ref{fgcutpt} (ii). 
%For a virtual knot diagram $D$ with a cut system $P$, the Gauss diagram of $D^*$ with a cut system $P^*$, $G(D^*, P^*)$ is obtained from $G(D, P)$ by reflection with respect to $y$-axis and revering all orientations of chords.
%For example, the Gauss diagram in Figure~\ref{fgexgauss1} (iii) is the Gauss diagram of the virtual knot diagram obtained from a diagram in Figure \ref{fig:extwtdiag} (i) by the reflection with respect to $y$-axis and switching all classical crossings. 
%We obtain the Gauss diagram of $\phi(D,P)$ by a local replacement of  each cup point $p$ of $G(D,P)$ and the corresponding cut point $p^*$ of $G(D^*. P^*)$ in $G(D, P)\amalg G(D^*. P^*)$ as in Figure~\ref{fggaussrep}. We denote the Gauss diagram of $\phi(D,P)$ by $G(\phi(D, P))$. The Gauss diagram in Figure~\ref{fgexgauss1} (iv) is the Gauss diagram of the converting diagram of the virtual link diagram with the cut system in Figure~\ref{fgcutpt} (ii). 
%
\begin{figure}[h]
\centering{
\includegraphics[width=7cm]{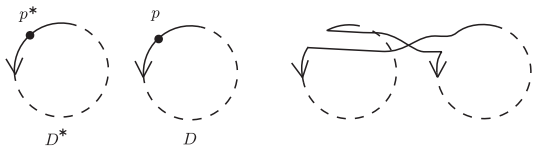} }
\caption{Replacement of  Gauss diagams}\label{fggaussrep}
\end{figure}
For the Gauss diagram of $(D\amalg D^*$, $P\amalg P^*)$, suppose that $p_1, \dots, p_n$ are points in $P$ such that the point $p_{i+1}$ follows the point $p_{i}$ along the orientation of $D$ and the point $p_i^*$ in $P^*$  corresponds to the point $p_i$. 
In what follows, we denote by the same symbol $p_i$ (or $p_i^*$) for the point $p_i$ of $P$ (or $p_i^*$ of $P^*$) and the corresponding point of $G(P)$ (or $G(P^*)$). 

Let $A_i$ (or $A_i^*$) be the arc of the Gauss diagram of $G(D\amalg  D^*, P\amalg P^*)$ between two points $p_i$ and $p_{i+1}$ (or $p_i^*$ and $p_{i+1}^*$), and  the arc between two points $p_n$ and $p_{1}$ (or $p_n^*$ and $p_{1}^*$) is $A_n$ (or $A_n^*$). Note that $A_i$ corresponds to $A_i^*$. We also denote an arc of $G(\phi(D,P))$ which corresponds to  $A_i$ or $A_i^*$  of $G(D\amalg D^*, P\amalg P^*)$ by $\widetilde{A_i}$ or $\widetilde{A_i^*}$, respectively. Here $\widetilde{A_i}$ (or $\widetilde{A_i^*}$) is the arc in $G(\phi(D,P))$ which is obtained from $A_i$ (or $A_i^*$) by removing a regular neighborhood of $p_i$ and $p_{i+1}$(or  $p_i^*$ and $p_{i+1}^*$).

\begin{lem}\label{lem3}
Let $(D,P)$ be a virtual knot diagram with  points $P$. Suppose $P=\{p_1,\dots, p_{2n}\}$ for a positive integer $n$. Then  $\phi(D,P)$ is a 2-component virtual link diagram $D_1\cup D_2$. Furthermore if an arc $\widetilde{A_i}$ is in $G(\phi(D, P))|_{D_1}$ (or $G(\phi(D, P))|_{D_2}$), then $\widetilde{A_{i+1}}$ and $\widetilde{A_{i}^*}$ are in $G(\phi(D, P))|_{D_2}$ (or $G(\phi(D, P))|_{D_1}$).
\end{lem}
\begin{proof}
We use the induction on $n$. 
Suppose that $n=1$, i.e. $D$ is a virtual knot diagram with 2 points $p_1$ and $p_2$. The Gauss diagram $G(\phi(D, P))$ is depicted as in Figure~\ref{fglemg1}, where the bold line and the thin line indicate the different components and we dropped all chords in the figure. In this case $\phi(D, P)$ is a 2-component virtual link diagram. Two arcs $\widetilde{A_1}$  and $\widetilde{A_2^*}$ are in one component of $G(\phi(D, P))$, and $\widetilde{A_{2}}$ and $\widetilde{A_{1}^*}$ are in the other.
\begin{figure}[h]
\centering{
\includegraphics[width=8cm]{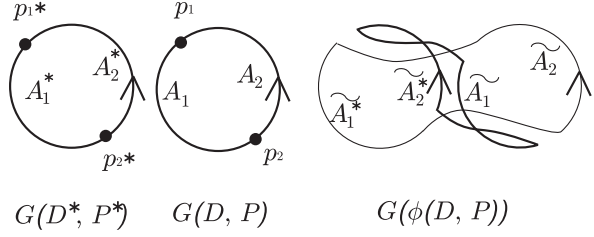} }
\caption{$G(\phi(D,P))$ of $D$ a pair of point}\label{fglemg1}
\end{figure}
Suppose that the statement is hold if the number of points is less than $2n$.
We assume that $D$ is a virtual knot diagram with $2n$ points $p_1,\dots, p_{2n}$. 
First, apply the replacement as in Figure~\ref{fggaussrep} to $2n-2$ points $p_1,\dots, p_{2n-2}$ and 
$p_1^*,\dots, p_{2n-2}^*$. Then we obtain a Gauss diagram $G$ with $4$ points $p_{2n-1}, p_{2n}$ and  $p_{2n-1}^*, p_{2n}^*$. 
By the hypothesis, the Gauss diagram $G$  is  depicted as in Figure~\ref{fglemg2} (i), where two arcs ${A_{2n-1}}$ and ${A_{2n}}$ (or two points $p_{2n-1}$ and $p_{2n}$) are in one component of $G$ and two arcs ${A_{2n-1}^*}$ and ${A_{2n}^*}$ (or two points $p_{2n-1}^*$ and $p_{2n}^*$) are in  the other. If an arc $\widetilde{A_i}$ is in  one component of $G$,  $\widetilde{A_i^*}$ (or $\widetilde{A_{i+1}}$) is in the other for  $i\ne 2n-1$ by the induction hypothesis.  
By applying the replacement in Figure~\ref{fggaussrep} to two pairs of points $p_{2n-1}$ and $p_{2n-1}^*$ and $p_{2n}$ and $p_{2n}^*$ of  the Gauss diagram $G$, we have a Gauss diagram as in Figure~\ref{fglemg2} (ii). Therefore we have the result.
\begin{figure}[h]
\centering{
\includegraphics[width=12cm]{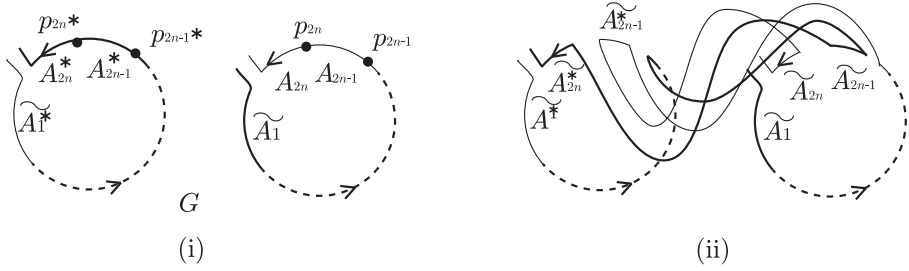} }
\caption{Replacement of  Gauss diagams}\label{fglemg2}
\end{figure}
\end{proof}
%

%Thus we have the following lemma
%\begin{lem}\label{thm5}
%Let $(D_1, P_1)$ and $(D_2, P_2)$ be  virtual knot diagrams with cut systems. If  $D_1$ and $D_2$ are equivalent (or K-equivalent), then the linking number of $\phi(D_2,P_2)$ is equal to that of $\phi(D_1,P_1)$.
%\end{lem}

\begin{proof}[Proof of Theorem~\ref{thmk1}]
Let $(D,P)$ be a virtual knot diagram with a cut system. By Lemma~\ref{lem3}, $\phi(D,P)$ is a 2-component virtual link diagram. Let $(D', P')$ be a virtual knot diagram with a cut system such that $D$ and $D'$ present the same virtual knot. 
By Theorem~\ref{thm2}, $\phi(D, P)$ and $\phi(D', P')$ are equivalent. Hence they have the same linking numbers.
\end{proof}
%For a 2-component link diagram $D$, we denote the linking number of $D$ by $\mathrm{lk}(D)$.
For a virtual knot $K$ and its diagram $D$ with the cut system $P$, we denote  the linking number of $\phi(D,P)$ by $\mathrm{lk}_N(K)$ or $\mathrm{lk}_N(D)$. 
Theorem~\ref{thmk2} states that $\mathrm{lk}_N(D)$ is equal to the odd writhe of $D$.

Let $D$ be a virtual knot diagram and $G$ be a Gauss diagram of $D$. 
For a classical crossing $c$, we denote by $\gamma_c$ the chord of $G$ corresponding to $c$. The endpoints of $\gamma_c$ divides the circle of $G$ into 2 arcs. We denote  the arcs by $I_c$ and $I'_c$ where $I_c$ is the arc which starts from the tail of $\gamma_c$ and terminates at the head. 
A classical crossing  $c$ of $D$ is said to be   {\it odd} if there are an odd number of endpoints of chords of $G$ on $I_c$. The {\it odd writhe} of $D$ is the sum of signs of odd crossings of $D$．Note that if a virtual knot diagram is normal, all classical crossings are not odd and hence the odd writhe is zero.

%\begin{thm}[\cite{rkauE}]\label{odd1}
It is shown in \cite{rkauE} that 
the odd writhe is an invariant of virtual knots.
%\end{thm}

\begin{proof}[Proof of Theorem~\ref{thmk2}]
Let $D$ be a virtual knot diagram and $P$ be a cut system of $D$.  
It is sufficient to show that 
odd crossings of $D$  correspond  to nonself classical crossings  of $\phi(D,P)$. 
Since $(D,P)$ admits an alternate orientation, the circle of the Gauss diagram $G(D,P)$  of $(D,P)$ admits an alternate orientation such that  one endpoint of each chord is 
a sink of the orientations and the other is a source. This implies the following condition: 
\begin{itemize}
\item[($\ast$)] 
For any classical crossing $c$ of $(D, P)$,
 the sum of the number of cut points on the arc $I_c$ and that of endpoints of chords appearing on $I_c$, is even. 
\end{itemize} 
%suppose that $p_1,\dots, p_n$ are points in $P$ such that the point $p_i $ follows the point $p_{i+1}$ along the orientation of $D$. 
Let $A_1, A_2, \dots $ be the arcs obtained by cutting the circle of $G(D,P)$ along the cut points.  We assume that 
 the arc $A_{i+1}$ appears after the arc $A_{i}$ along the orientation of $D$. 
We also denote   by $\widetilde{A_i}$ the arc of $G(\phi(D,P))|_D$ which corresponds to $A_i$. 
For a classical  crossing $c$ of $(D, P)$, the classical crossing corresponding to $c$ in $\phi(D, P)|_D$ is denoted by $\tilde{c}$. 
Let $c$ be  an odd crossing of $(D, P)$. 
Suppose that  one endpoint of  $\gamma_c$  
is on $A_k$ and the other endpoint  is on $A_j$. By definition, there are an odd number of endpoints of chords on $I_c$ in  $G(D,P)$. 
By the condition ($\ast$) above, we see that there are an odd number of cut points on $I_c$ in  $G(D,P)$.  
Then we see that $k$ is not congruent to  $j$ modulo $2$.  By Lemma~\ref{lem3} $\widetilde{A_k}$ and $\widetilde{A_j}$ in  $G(\phi(D,P))|_D$ are in distinct components of $G(\phi(D,P))$. Thus the classical crossing $\tilde{c}$ is a nonself classical crossing of $\phi(D,P)$. By a similar argument, we see that if $c$ is not an odd crossing of $(D,P)$, then $\tilde{c}$ is a self classical crossing of $\phi(D,P)$.
\end{proof}

We show some properties of $\mathrm{lk}_N(D)$. They are also obtained from the properties of the odd writhe.

\begin{cor}[\cite{rkauE}]\label{prop3}
Let $K$ be a normal virtual knot. Then $\mathrm{lk}_N(K)$ is zero.
\end{cor}
\begin{proof}
Let $D_N$ be a normal knot diagram of $K$. 
The empty set $\emptyset$ is a cut system.
The virtual link diagram $\phi(D_N,\emptyset)$ is the disjoint union, $D_N \amalg D_N^{*}$ of $D_N$ and $D^*_N$. Thus $\mathrm{lk}_N(K)$ is zero.
\end{proof}
Let $D$ be a virtual knot diagram. The virtual knot diagram obtained from $D$ by switching the over-under information of all classical crossing (or by reflection) is denoted by $D^{\sharp}$ (or $D^{\dag}$).
\begin{cor}[\cite{rkauE}]\label{prop8}
Let $D$ be a virtual knot diagram. If $\mathrm{lk}_N(D)$ is not zero, then $D$ is not equivalent to $D^{\sharp}$ (or $D^{\dag})$. 
\end{cor}
\begin{proof}
It is clear that  $\mathrm{lk}_N(D^{\sharp})=\mathrm{lk}_N(D^{\dag})=-\mathrm{lk}_N(D)$. This implies the result.
\end{proof}

For example, the virtual knot presented by the diagram $D$ in Figure~\ref{fgexapli1}, is not normal by Corollary~\ref{prop3}, since $\mathrm{lk}_N(D)=-2$. By Corollary~\ref{prop8},  $D$ is not equivalent to $D^{\sharp}$ (or $D^{\dag})$.
\begin{figure}[h]
\centering{
\includegraphics[width=12cm]{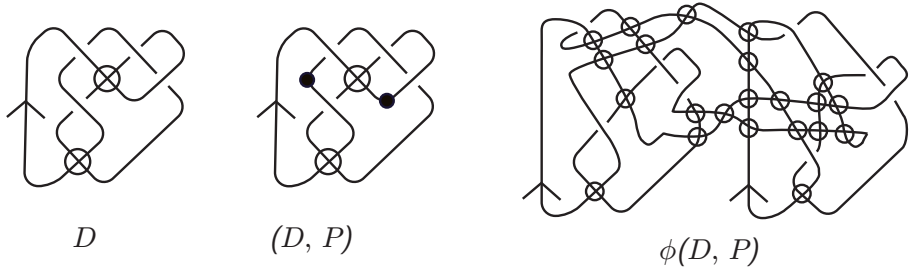} }
\caption{A virtual knot which is not normal}\label{fgexapli1}
\end{figure}
\subsection{Linking invariants of even virtual links}
H. Miyazawa K. Wada and A. Yasuhara defined an invariant of even virtual links by use of the virtual orientation. 
We give another definition of their invariant. 

Let $D$ be a virtual link diagram. $D$ is said to be {\it even}, if on each circle of a Gauss diagram of $D$ there is an even number of endpoints of chords. A virtual knot diagram is even. If a virtual link diagram $D$ is even, any virtual link diagram which is equivalent to $D$ is even. 
Note that when we go around  each component of an even virtual link diagram, we meet an even number of virtual crossings. A normal virtual link diagram is even, and any a virtual link diagram which presents a normal virtual link is even.

Let $D$ be an ordered unoriented even $r$-component virtual link diagram such that $D=D^1\cup \dots \cup D^r$ where $D^i$ is a component of $D$. The {\it virtual orientation} is an orientation of $D$ such that 
they switch an opposite direction at each virtual crossing of $D$ as in Figure~\ref{fg:virorientation}. 
(The orientation is not switched at classical crossings.)

\begin{figure}[h]
\centering{
\includegraphics[width=4cm]{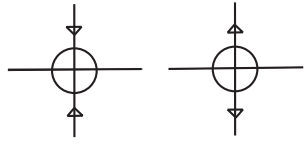} }
\caption{Virtual orientation}\label{fg:virorientation}
\end{figure}

We give $D$ a virtual orientation ${\cal O}_v(D)$. 
The sum of signs of classical crossings between $D^i$ and $D^j$ whose overpass of $D^i$ is denoted by $\lambda_{D}(i,j)$ under the virtual orientation ${\cal O}_v(D)$. 

\begin{thm}[\cite{rMWY}]\label{thmLinkInv}
$|\lambda_{D}(i,j)|$ is an invariant of ordered unoriented even virtual links.
\end{thm}

\begin{prop}
The number of cut points of each component of an even virtual link diagram with a cut system is even.
\end{prop}
\begin{proof}
Let $(D,P_0)$ be even virtual link diagram $D$ with the canonical cut system $P_0$. 
Each component of $D$ has an even number of cut points, since we meet an even number virtual crossings when we go around each component of it. Any cut system of $D$ is obtained from $P_0$ by 
cut point moves I, II and III. Then we have the same result in another cut system of $D$ since by a cut point move II or  III  changes the number of cut points on a component of $D$ by $\pm 2$ or $\pm 4$. 
\end{proof}

Let $D$ be an even  $r$-component virtual link diagram such that $D=D^1\cup \dots \cup D^r$ where $D^i$ is a component of $D$. Suppose that $P$ is a cut system of $D$.
A {\it cut orientation} is an orientation of $(D, P)$ such that they switch an opposite direction at each cut point of $D$ as in Figure~\ref{fg:cutorientation}. We denote it ${\cal O}_c(D,P)$. (For a cut orientation, the orientation is not switched at classical crossings and virtual crossings.)

\begin{figure}[h]
\centering{
\includegraphics[width=4cm]{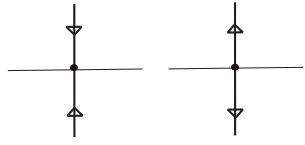} }
\caption{Cut orientation}\label{fg:cutorientation}
\end{figure}

Let $(D,P)$ be even virtual link diagram $D$ with a cut system $P$ such that $D=D^1\cup \dots \cup D^r$ where $D^i$ is a component of $D$. 
We give $D$ a cut orientation ${\cal O}_c(D,P)$. 
The sum of signs of classical crossings between $D^i$ and $D^j$ whose overpass of $D^i$ is denoted by $\nu_{(D, P)}(i,j)$ under the cut orientation ${\cal O}_c(D,P)$. 

\begin{prop}\label{cutori1}
For two cut systems $P$ and $P'$ of $D$, let ${\cal O}_c(D,P)$ (or ${\cal O}_c(D,P')$) be a cut orientation 
of $(D,P)$ (or $(D,P')$). 
Then we have $|\nu_{(D, P)}(i,j)|=|\nu_{(D, P')}(i,j)|$.
\end{prop}
\begin{proof}
Suppose that $P'$ is obtained form $P$ by one of cut point moves I, II and III to a component $D^i$ (or two  components $D^i$ and $D^k$) of $D$ as in Figure~\ref{fg:cutoricutmove} (i), (ii) (or (iii)). 
The cases of a cut point moves I (or II), we give a cut orientation ${\cal O}_c(D,P)$ to $(D, P)$ as in Figure~\ref{fg:cutoricutmove} (i) (or  (ii)) and ${\cal O}_c(D,P')$ to $(D, P')$ as in Figure~\ref{fg:cutoricutmove} (i) (or  (ii))  (a) or (b).
In both cases, if a sign of each classical crossing $c$ between $D^i$ and $D^j$ of $D$ with respect to ${\cal O}_c(D,P)$ is $\epsilon$,  the sign of $c$ with respect to ${\cal O}_c(D,P')$ is $\pm\epsilon$.
The case of a cut point moves III, we give a cut orientation ${\cal O}_c(D,P)$ to $(D, P)$ as in Figure~\ref{fg:cutoricutmove} (iii)  and ${\cal O}_c(D,P')$ to $(D, P')$ as in Figure~\ref{fg:cutoricutmove} (iii) (a), (b), (c) or (d).
In each case, if a sign of each classical crossing $c$ between $D^i$ and $D^j$ of $D$ with respect to ${\cal O}_c(D,P)$ is $\epsilon$,  the sign of $c$ with respect to ${\cal O}_c(D,P')$ is $\pm \epsilon$.
Thus we have the result.

\begin{figure}[h]
\centering{
\includegraphics[width=14cm]{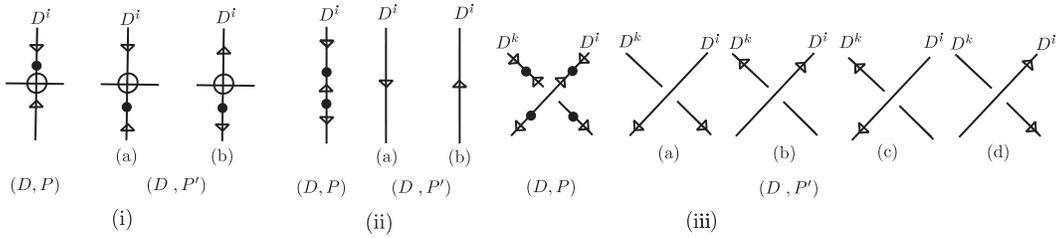} }
\caption{Cut orientation  and cut moves}\label{fg:cutoricutmove}
\end{figure}

\end{proof}

We denote $|\nu_{(D, P)}(i,j)|$ by $|\nu_{D}(i,j)|$. 
\begin{thm}\label{thmLinkInvCut}
$|\nu_{D}(i,j)|$ coincides to $|\lambda_{D}(i,j)|$. Thus 
$|\nu_{D}(i,j)|$ is an invariant of ordered unoriented virtual link.
\end{thm}

\begin{proof}
Let $D$ be an ordered unoriented even  $r$-component virtual link diagram such that $D=D^1\cup \dots \cup D^r$ where $D^i$ is a component of $D$. 
We give a virtual orientation of $D$, ${\cal O}_v(D)$ to $D$. Let $P_0$ be the canonical cut system of $D$. There is a cut orientation of $(D, P_0)$ 
whose orientation of each arc of $D$  corresponds to that of $D$ in ${\cal O}_v(D)$ as in Figure~\ref{fgvircutcors}.
\end{proof}
\begin{figure}[h]
\centering{
\includegraphics[width=6cm]{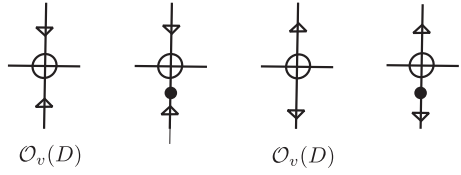} }
\caption{Virtual orientation  and cut orientation}\label{fgvircutcors}
\end{figure}

We give an interpretation of the linking invariant $|\lambda_{D}(i,j)| (=|\nu_{D}(i,j)|)$ in terms of our covering method. 
Let $(D, P)$ be an even  $r$-component virtual link diagram with a cut system $P$ such that $D=D^1\cup \dots \cup D^r$ where $D^i$ is a component of $D$. 
We give an cut orientation ${\cal O}_c(D,P)$ to $D$. Let $(D^*, P^*)$ be a copy of $(D, P)$ and ${\cal O}_c(D^*,P^*)$ be the cut orientation of $(D^*, P^*)$ such that an orientation of each arc of $(D^*, P^*)$ is different from the corresponding arc of $(D, P)$ in ${\cal O}_c(D,P)$.
% Put them in the position for coherent double covering diagram. We see an example of two virtual link diagrams in Figure~\ref{fgconvert3} (i) with cut orientations. 
We take the coherent double covering diagram of $(D, P)$ in the same manner as in Figure~\ref{fgconvert2}. We call this the {\it coherent double covering diagram with a cut orientation} ${\cal O}_c(D,P)$ denoted by $\phi_{\cal{O}_c(D,P)}(D,P)$. The virtual link diagram in Figure~\ref{fgconvert3} (ii) is the coherent double covering diagram of the virtual link diagram in Figure~\ref{fgconvert3} (i) with a cut orientation $\cal{O}_c(D,P)$.
Note that the orientation of the coherent double covering diagram derived from  cut orientations  ${\cal O}_c(D,P)$ and ${\cal O}_c(D^*,P^*)$.

\begin{figure}[h]
\centering{
\includegraphics[width=5.5cm]{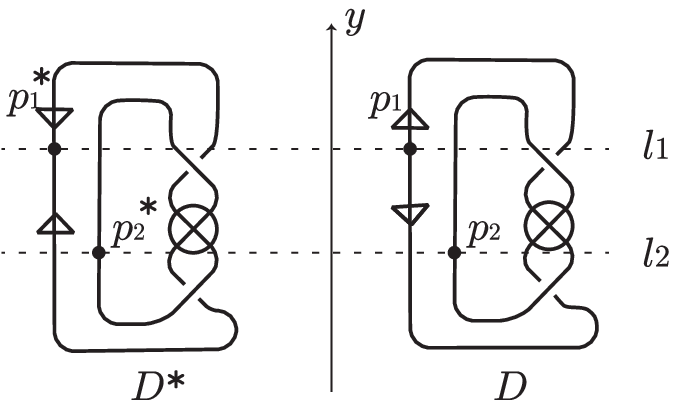}\hspace{1cm}
\includegraphics[width=4.8cm]{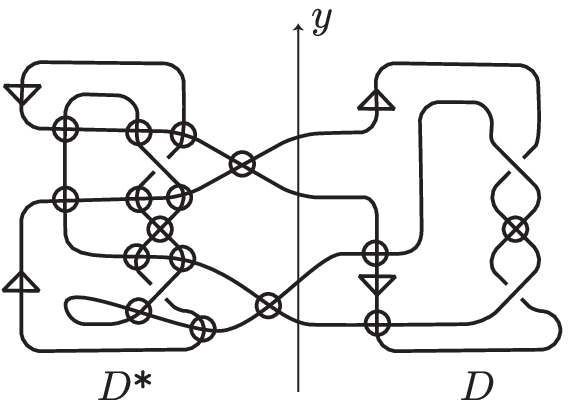}}\\
\centerline{(i)\hspace{6cm}(ii)}
%\caption{A virtual link diagram and it's refrection}
\caption{The coherent double covering of a virtual link diagram with a cut orientation}\label{fgconvert3}
\end{figure}

Let $(D, P)$ be an ordered unoriented even  $r$-component virtual link diagram with a cut system $P$ such that $D=D^1\cup \dots \cup D^r$ where $D^i$ is a component of $D$. We denote the subset of the cut system $P$ whose points are on arcs of $D^i$ by $P^i$.
A component of the Gauss diagram of $(D,P)$ corresponding a component $D^i$ with $P^i$ by $G(D^i,P^i)$. 
For the Gauss diagram of $(D\amalg D^*$, $P\amalg P^*)$, suppose that $p_1^i, \dots, p_{2n_i}^i$ (or $(p_1^i)^*, \dots, (p_{2n_i}^i)^*$) are points in $P^i$ such that the point $p_{j+1}^i$ follows the point $p_{j}^i$ along  $D^i$ and the point $(p_j^i)^*$ in $(P^i)^*$  coressponds to the point $p_j^i$. 
In what follows, we denote by the same symbol $p_j^i$ (or $(p_i^i)^*$) for the point $p_j^i$ of $P^i$ (or $(p_j^i)^*$ of $(P^i)^*$) and the corresponding point of $G(P)$ (or $G(P^*)$).

Let $A_j^i$ (or $(A_j^i)^*$) be the arc of the Gauss diagram of $G(D\amalg  D^*, P\amalg P^*)$ between two points $p_j^i$ and $p_{j+1}$ (or $(p_j^i)*$ and $(p_{j+1}^i)^*$), and  the arc between two points $p_{2n_i}^i$ and $p_{1}^i$ (or $(p_{n_i}^i)^*$ and $(p_{1}^i)^*$) is $A_{2n_i}^i$ (or $(A_{2n_i}^i)^*$). Note that $A_j^i$ corresponds to $(A_j^i)^*$. We also denote an arc of $G(\phi(D,P))$ which corresponds to  $A_j^i$ (or $(A_j^i)^*$)  of $G(D\amalg D^*, P\amalg P^*)$ by $\widetilde{A_j^i}$ (or $\widetilde{(A_j^i)^*}$), respectively. 
%Here $\widetilde{A_j^i}$ (or $\widetilde{(A_j^i)^*}$) is the arc in $G(\phi_{\cal{O}_c(D,P)}(D,P))$ which is obtained from $A_j^i$ (or $(A_j^i)^*$) by removing a regular neighborhood of $p_j^i$ and $p_{j+1}^i$(or  $(p_j^i)^*$ and $(p_{j+1}^i)^*$).

%
\begin{lem}\label{lem4}
Let $(D, P)$ be an even  $r$-component virtual link diagram with a cut system $P$ such that $D=D^1\cup \dots \cup D^r$ where $D^i$ is a component of $D$. 
Then  $\phi(D,P)$ is a 2$r$-component virtual link diagram such that  $\phi(D,P)|_{D^i}$ is a 2-component link diagram $D_1^i\cup D_2^i$. Furthermore if an arc $\widetilde{A_j^i}$ is in  $G(\phi(D, P))|_{D_1^i}$ (or $G(\phi(D, P))|_{D_2^i}$), then $\widetilde{A_{j+1}^i}$ and $\widetilde{(A_{j}^i)^*}$ are in $G(\phi(D, P))|_{D_2^i}$ 
(or $G(\phi(D, P))|_{D_1^i}$).
\end{lem}
\begin{proof}
The proof  is obtained by a similar argument as in Lemma~\ref{lem3}. 
\end{proof}

Let $(D, P)$ be an ordered unoriented even  $r$-component virtual link diagram with a cut system $P$ such that $D=D^1\cup \dots \cup D^r$ where $D^i$ is a component of $D$. We denote the subset of the cut system $P$ whose points are on arcs of $D^i$ by $P^i$.
For a classical  crossing $c^{i,j}$ between $D^i$ and $D^j$ of $(D, P)$, the classical crossing corresponding to $c^{i,j}$ in $\phi_{\cal{O}_c(D,P)}(D, P)$ between ${D_k^i}$ and ${D_l^k}$ is denoted by $\widetilde{c^{i,j}_{k,l}}$ ($k,l\in \{1,2\}$). 

Let $c^{i,j}$ between $D^i$ and $D^j$ of $(D, P)$ such that an arc of $D^i$ is an overpass of $c^{i,j}$. Then the overpass of  a classical crossing $\widetilde{c^{i,j}_{k_1,l_1}}$ in $\phi_{\cal{O}_c(D,P)}(D, P)$ corresponding $c^{i,j}$,  is in  ${D_{k_1}^i}$. For the other classical crossing $\widetilde{c^{i,j}_{k_2,l_2}}$ corresponding $c^{i,j}$,  its overpass is in ${D_{k_2}^i}$ where $k_2 \ne k_1$ and $l_2\ne l_1$ by Lemma~\ref{lem4}. Then we have the following theorem.

\begin{thm}\label{prop18}
Let $(D, P)$ be an ordered unoriented even  $r$-component ordered unoriented virtual link diagram with a cut system $P$ such that $D=D^1\cup \dots \cup D^r$ where $D^i$ is a component of $D$. For a cut orientation $\cal{O}_c(D,P)$ of $(D, P)$, the absolute value of the sum of the classical crossings between $D_k^i$ and $D_{l_1}^j\cup D_{l_2}^j$ whose overpasses are in $D^i_k$ in  the coherent double covering diagram $\phi_{\cal{O}_c(D,P)}(D, P)$ with a cut orientation $\cal{O}_c(D,P)$ of $(D, P)$, coincides to $|\nu_D(i,j)|$.
\end{thm}

\begin{proof}
Let $c^{i,j}$ be a classical crossing of $D$ between $D^i$ and $D^j$ whose  overpass is in $D^i$. One of two  classical crossings of $\phi_{\cal{O}_c(D,P)}(D, P)$ corresponding to $c^{i,j}$  is between $D_k^i$ and $D_{l_1}^j\cup D_{l_2}^j$, and its overpass  is in $D_k^i$. The sign of such a classical crossing is equal to that of $c^{i,j}$.
\end{proof}

For an ordered unoriented even virtual link diagram $D=D^1\cup D^2$ with a virtual orientation in Figure~\ref{fgexlinkinv1} (i), $|\lambda_D(1,2)|=|\lambda_D(2,1)|=2$. In Figure \ref{fgexlinkinv1} (ii), the diagram D with a cut system $P$ and a cut orientation $\cal{O}_c(D,P)$ and an coherent double covering diagram of $(D,P)$  with a cut orientation $\cal{O}_c(D,P)$, $\phi_{\cal{O}_c(D,P)}(D, P)={D_1^1}\cup {D_2^1}\cup{D_1^2}\cup{D_2^2}$ are depicted, where ${D_1^i}\cup {D_2^i}$ correspond to a component $D^i$ of $D$. The absolute value of the sum of the classical crossings between $D_1^i$ (or $D_2^i$) and $D_{1}^j\cup D_{2}^j$ whose overpasses belong to $D_1^i$ (or $D_2^i$) coincides to $|\lambda_D(i,j)|$ for $i \ne j$.

\begin{figure}[h]
\centering{
\includegraphics[width=9.5cm]{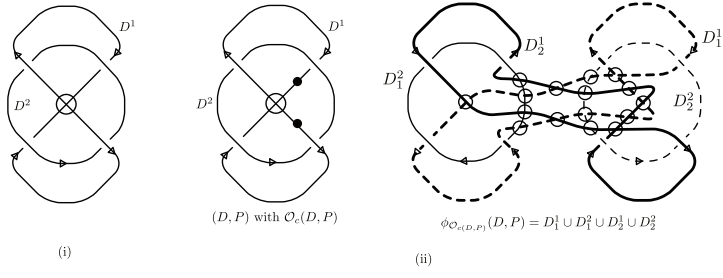}}
%\caption{A virtual link diagram and it's refrection}
\caption{Example : Linking invariant and  coherent double covering of a virtual link diagram with a cut orientation}\label{fgexlinkinv1}
\end{figure}

Let $(D, P)$ be an ordered oriented even  $r$-component virtual link diagram with a cut system $P$ such that $D=D^1\cup \dots \cup D^r$ where $D^i$ is a component of $D$.
For a cut system $P$ of $D$,   we denote a $2$-component sub-diagram $\phi (D, P)|_{D^i}$ of  the coherent double covering diagram $\phi (D, P)$ which corresponds to $D^i$,  by ${D^i_1}\cup{D^i_2}$. 
Let 
$\tilde{\text{lk}}(D^i_k,D^j_l)$ be 
the linking number between ${D}^i_k$ and ${D}^j_l$ for $k,l=1,2$.

The set $\{\tilde{\text{lk}}(D^i_1,D^j_1), \tilde{\text{lk}}(D^i_1,D^j_2)\}$  is denoted by $Q_{ij}(D)$($i\ne j$). For a virtual knot diagram $D$, $\mathrm{lk}_N(D)$ is $\tilde{\text{lk}}(D_1,D_2)$, where $\phi(D,P)=D_1\cup D_2$.

\begin{cor}
The set $Q_{ij}(D)$  is an invariant of ordered oriented even virtual links. Furthermore 
$Q_{ij}(D)$ is equal to the set $\{\tilde{\mathrm{lk}}(D^i_2,D^j_1),\tilde{\mathrm{lk}}(D^i_2,D^j_2)\}$ and $\{\tilde{\mathrm{lk}}(D^i_1,D^j_1),\tilde{\mathrm{lk}}(D^i_2,D^j_1)\}$, $\{\tilde{\mathrm{lk}}(D^i_1,D^j_2),\tilde{\mathrm{lk}}(D^i_2,D^j_2)\}$.
The linking number $\tilde{\mathrm{lk}}(D^i_1,D^i_2)$ is an invariant of an ordered oriented even virtual link.
\end{cor}

\begin{proof}
Let $(D, P)$ be an ordered oriented even  $r$-component virtual link diagram with a cut system $P$ such that $D=D^1\cup \dots \cup D^r$ where $D^i$ is a component of $D$.
Let $c^{i,j}$ be a classical crossing between $D^i$ and $D^j$ of $(D, P)$. If one calssical crossing $\widetilde{c^{i,j}_{k_1,l_1}}$ in $\phi(D, P)$ corresponding $c^{i,j}$ is between  ${D_{k_1}^i}$ and ${D_{l_1}^j}$ and, then the other is a classical crossing $\widetilde{c^{i,j}_{k_2,l_2}}$ between ${D_{k_2}^i}$ and ${D_{l_2}^j}$ for $k_1\ne k_2$ and $l_1\ne l_2$. The sings of two classical crossings $\widetilde{c^{i,j}_{k_1,l_1}}$ and $\widetilde{c^{i,j}_{k_2,l_2}}$ are the same. 
Then we have $\tilde{\text{lk}}(D^i_1,D^j_1)=\tilde{\text{lk}}(D^i_2,D^j_2)$ and $\tilde{\text{lk}}(D^i_1,D^j_2)=\tilde{\text{lk}}(D^i_2,D^j_1)$. 
\end{proof}

Let  an ordered oriented even virtual link diagrams $D=D^1\cup D^2$ with a cut system ($P=\emptyset$) and  $D'={D^1}'\cup {D^2}'$ with a cut system $P'$ be as in Figure~\ref{fgexlinking1} (i). In Figure~\ref{fgexlinking1} (ii) we show a coherent double covering diagram 
$\phi(D', P')={D_1^1}'\cup {D_2^1}'\cup{D_1^2}'\cup{D_2^2}'$.
We have $\tilde{\text{lk}}(D^1_1,D^1_2)=0$ and $\tilde{\text{lk}}({D^1_1}',{D^1_2}')=2$. So $D$ is not equivalent to $D'$.

\begin{figure}[h]
\centering{
\includegraphics[width=9.5cm]{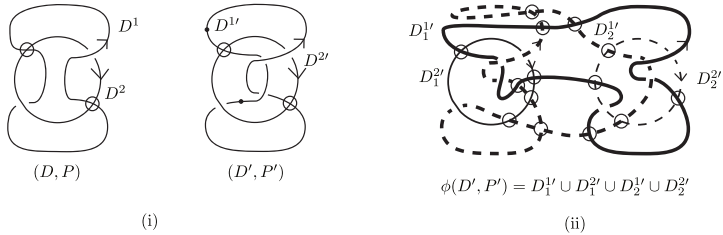}}
%\caption{A virtual link diagram and it's refrection}
\caption{Example : Linking number of  coherent double covering of a virtual link diagram}\label{fgexlinking1}
\end{figure}

\vspace{5pt}
%\vspace{3zw}
\noindent
{\bf Acknowledgement}\\
The author  would like to thank Seiichi Kamada for his useful suggestion.

 \end{document}